\documentclass{amsart}
\usepackage{amssymb,setspace,diagrams,fixltx2e}
\usepackage[l2tabu,orthodox]{nag}
\usepackage[headings]{fullpage}
\usepackage[colorlinks]{hyperref}
\newcommand{\QQ}{\mathbb Q}
\newcommand{\FF}{\mathbb F}
\newcommand{\ZZ}{\mathbb Z}

\newcommand{\EE}{\mathbb E}
\newcommand{\BB}{\mathbb B}
\newcommand{\DD}{\mathbb D}
\renewcommand{\AA}{\mathbb A}
\newcommand{\NN}{\mathbb N}
\newcommand{\Brig}{\BB^+_{\mathrm{rig},\QQ_p}}

\DeclareMathOperator{\Sel}{Sel}
\DeclareMathOperator{\Tr}{Tr}
\DeclareMathOperator{\coker}{coker}
\DeclareMathOperator{\Hom}{Hom}

\DeclareMathOperator{\Gal}{Gal}

\DeclareMathOperator{\rank}{rank}

\DeclareMathOperator{\Iw}{Iw}

\DeclareMathOperator{\cris}{cris}
\DeclareMathOperator{\rig}{rig}
\DeclareMathOperator{\Col}{Col}

\DeclareMathOperator{\cor}{cor}
\DeclareMathOperator{\Rep}{Rep}
\DeclareMathOperator{\Fil}{Fil}

\newcommand{\Ep}{E_{p^\infty}}

\newcommand{\Qp}{\QQ_p}
\newcommand{\Zp}{\ZZ_p}

\newcommand{\calG}{\mathcal{G}}

\newcommand{\calC}{\mathcal{C}}
\newcommand{\calO}{\mathcal{O}}
\newcommand{\CC}{\mathbb{C}}

\newcommand{\vp}{\varphi}
\newcommand{\MHG}{{\mathfrak M}_H(G)}

\newtheorem{theorem}{Theorem}[section]
\newtheorem{proposition}[theorem]{Proposition}
\newtheorem{lemma}[theorem]{Lemma}
\newtheorem{corollary}[theorem]{Corollary}
\newtheorem{remark}[theorem]{Remark}
\newtheorem{assumption}[theorem]{Assumption}

\newtheorem{definition}[theorem]{Definition}
\newtheorem{conjecture}[theorem]{Conjecture}

\begin{document}  

\title{Signed Selmer groups over $p$-adic Lie extensions}

\author{Antonio Lei}
\address[Lei]{School of Mathematical Sciences\\
Monash University\\
Clayton, VIC 3800\\
Australia}
\email{antonio.lei@monash.edu}
\thanks{The first author is supported by an ARC DP1092496 grant.}

\author{Sarah Livia Zerbes}
\address[Zerbes]{Department of Mathematics\\
  Harrison Building\\
  University of Exeter\\
  Exeter EX4 4QF, UK
}
\email{s.zerbes@exeter.ac.uk}

\thanks{The second author is supported by EPSRC Postdoctoral Fellowship EP/F043007/1.}

\begin{abstract}
 Let $E$ be an elliptic curve over $\QQ$ with good supersingular reduction at a prime $p\geq 3$ and $a_p=0$. We generalise the definition of Kobayashi's plus/minus Selmer groups over $\QQ(\mu_{p^\infty})$ to $p$-adic Lie extensions $K_\infty$ of $\QQ$ containing $\QQ(\mu_{p^\infty})$, using the theory of $(\vp,\Gamma)$-modules and Berger's comparison isomorphisms. We show that these Selmer groups can be equally described using the ``jumping conditions" of Kobayashi via the theory of overconvergent power series. Moreover, we show that such an approach gives the usual Selmer groups in the ordinary case. 
\end{abstract}


\maketitle
\tableofcontents

 \section{Introduction}
 
  Let $E$ be an elliptic curve defined over $\QQ$ with good supersingular reduction at a prime $p\geq 3$ and $a_p=0$. Kobayashi \cite{kobayashi03} constructed two $\Lambda$-cotorsion Selmer groups $\Sel^i(E/\QQ(\mu_p^{\infty}))$, $i=1,2$ (denoted by $\Sel^\pm(E/\QQ(\mu_p^{\infty}))$ in \emph{op.cit.}) by modifying the local condition at $p$ in the definition of the usual Selmer group. In this paper, we propose an analogous definition of signed Selmer groups $\Sel^i(E\slash K_\infty)$ of $E$ over $K_\infty$ for $i=1,2$, where $K_\infty$ is a $p$-adic Lie extension over $\QQ$ which contains $\QQ(\mu_{p^\infty})$.

  The main idea of our construction is the use of Berger's comparison isomorphism in~\cite{berger02}. Let us first recall the description of the signed Selmer groups $\Sel^i(E\slash \QQ(\mu_{p^\infty}))$ in terms of $p$-adic Hodge theory as given in \cite{leiloefflerzerbes10}. Let $V=\QQ_p\otimes T$ where $T=T_pE$ is the Tate module of $E$ at $p$, then $V$ is a crystalline representation of $\calG_{\QQ_p}$, the absolute Galois group of $\Qp$. We write $\NN(T)$ for the Wach module of $T$ (c.f. \cite{berger03,wach96}). Then a result of Fontaine/Berger states that we have a canonical isomorphism $H^1_{\Iw}(\QQ_p,T)\cong \NN(T)^{\psi=1}$ (we will identify these two objects throughout the paper). Let $n^\pm$ be the canonical basis of $\NN(T)$ as constructed in the appendix of \emph{op.cit.}, and let $v^\pm$ be the induced basis of $\DD_{\cris}(V)$. Via Berger's comparison isomorphism, any element $x\in\NN(T)^{\psi=1}$ can be expressed in the form $x=x_1 v^++x_2v^-$ where $x_i\in\BB^+_{\rig,\QQ_p}$ for $i=1,2$. Define
   \[ H^1_{\Iw}(\QQ_p,T)^i=\left\{ x\in \NN(T)^{\psi=1}\mid \vp(x_i)=-p\psi(x_i)\right\},\]
  and let $H^1(\QQ_p(\mu_{p^n}),T)^i$ be the image of $H^1_{\Iw}(\QQ_p,T)^i$ under the natural projection map $H^1_{\Iw}(\QQ_p,T)\rightarrow H^1(\QQ_p(\mu_{p^n}),T)$. Define $H^1_{f,i}(\QQ_p(\mu_{p^n}),\Ep)$ to be the exact annihilator of $H^1(\QQ_p(\mu_{p^n}),T)^i$ under the Tate pairing. One then defines $\Sel^i(E\slash \QQ(\mu_{p^n}))$ by replacing the usual local condition $H^1_{f}(\QQ_p(\mu_{p^n}),\Ep)$ at the unique prime of $\QQ(\mu_{p^n})$ above $p$ by $H^1_{f,i}(\QQ_p(\mu_{p^n}),\Ep)$.
  
  If $F$ is an arbitrary finite extension of $\QQ_p$, then $H^1_{\Iw}(F,T)$ is canonically isomorphic to $\DD_F(T)^{\psi=1}$, where $\DD_F(T)$ denotes the $(\varphi,\Gamma)$-module of $T$ over the base field $\AA_F$. Moreover, every element $x\in \DD_F(T)$ can be uniquely written as $x=x_1v^++x_2v^-$ with $x_i\in\BB^\dagger_{\rig,F}$. It therefore seems natural to make the following definition: for $i=1,2$, let 
  \[ H^1_{\Iw}(F,T)^i=\left\{ x\in \DD_F(T)^{\psi=1}\mid \vp(x_i)=-p\psi(x_i)\right\}.\]
  One can the repeat the above construction to define `new' local conditions $H^1_{f,i}(F(\mu_{p^n}),\Ep)$ for $i=1,2$. If $K$ is a finite extension of $\QQ$, this allows us to define signed Selmer groups $\Sel^i(E\slash K(\mu_{p^n}))$ for $i=1,2$ and for all $n\geq 0$. By passing to the direct limit over $n$, we obtain the Selmer groups $\Sel^i(E\slash K_\infty)$. The details of this construction is given in Section~\ref{supersingular}.

  When $E$ has good ordinary reduction at $p$, we have defined $\Sel^i(E/\QQ(\mu_{p^\infty}))$ for $i=1,2$ in \cite{leiloefflerzerbes10} in the same way as the good supersingular case. To justify the proposed definition of signed Selmer groups over $K_\infty$, we show in Section~\ref{ordinary} that on extending our construction to the good ordinary case, $\Sel^2(E\slash K_\infty)$ again agrees with the usual Selmer group $\Sel(E\slash K_\infty)$ for any finite extension $K$ of $\QQ$.

  In Section~\ref{alter}, we give a more explicit description of the local conditions we use to define the signed Selmer groups in the supersingular case. If $F$ is a finite extension of $\Qp$, we write $F_n=F(\mu_{p^n})$. We define for a large integer $N$, which depends on $F$,
  \[
   \hat{E}_N^i(\calO_{F_n})   := \left\{x\in\hat{E}(\calO_{F_n}): \Tr_{F_n/F_m}x\in \hat{E}(\calO_{F_{m-1}}))\text{ for all }m\in S_{N,i'}^n \text{ and }\Tr_{F_n/F_N}x=0\right\},
  \]
  where
  \begin{align*}
   S_{N,1'}^n&=\{m\in[N+1,n]:m\text{ even}\};\\
   S_{N,2'}^n&=\{m\in[N+1,n]:m\text{ odd}\}.
  \end{align*}
  We show that if we define $\Sel_N^{(i)}(E/K_n)$ by replacing the local conditions at places above $p$ in the definition of $\Sel(E/K_n)$ by these ``jumping conditions", then  for $i=1,2$ we have isomorphisms
  \[
   \Sel_N^{(i)}(E/K_\infty)\cong\Sel^{i}(E/K_\infty)
  \]
  on taking direct limits.

  In Section~\ref{MHG}, we extend the definition of signed Selmer groups to $p$-adic Lie extensions and formulate a $\mathfrak{M}_H(G)$-conjecture, analogous to the one for the good ordinary case in \cite{cfksv}. Finally, we will explain some of the difficulties we encountered when attempting to extract information on the conjecture in Section~\ref{cyclotomicshortexact}.
  
  \vspace{1ex}
  
  {\it Acknowledgements.} We would like to thank John Coates and David Loeffler for their interest, and the latter for many helpful comments. Part of this paper was written while the authors were visiting the University of Warwick; they would like to thank the number theory group for their hospitality.

  
 \section{Notation and background}  
 
  Let $F$ be a finite extension of $\QQ_p$. Write $\Rep_{\ZZ_p}(\calG_F)$ (resp. $\Rep_{\QQ_p}(\calG_F)$) for the category of finitely generated $\ZZ_p$-modules (resp. finite-dimensional $\QQ_p$-vector spaces) with a continuous action of $\calG_F$. 

  For an integer $n\geq 1$, we write $\QQ_{p,n}=\QQ_p(\mu_{p^n})$, $\QQ_{p,\infty}=\varinjlim \QQ_{p,n}$ and $\Gamma=\Gal(\QQ_{p,\infty}\slash\QQ_p)$. More generally, if $F$ is a finite extension of $\QQ_p$, we write $F_n=F(\mu_{p^n})$, $F_\infty=\varinjlim F_n$, $H_F=\Gal(\overline{\QQ_p}\slash F_\infty)$ and $\Gamma_F=\Gal(F_{\infty}\slash F)$. For $T\in\Rep_{\ZZ_p}(\calG_F)$, define $H^1_{\Iw}(F,T)=\varprojlim H^1(F_n,T)$ where the connecting maps are corestrictions $\cor_{n/m}:H^1(F_n,T)\rightarrow H^1(F_m,T)$ for $n\ge m$. If $V\in\Rep_{\QQ_p}(\calG_F)$, let $H^1_{\Iw}(F,V)=H^1_{\Iw}(F,T)\otimes_{\ZZ_p}\QQ_p$, where $T$ is a $\calG_F$-invariant lattice of $V$. If $G$ is a compact $p$-adic Lie group, we write $\Lambda(G)=\ZZ_p[[G]]$ for its completed group ring over $\ZZ_p$. 
 
  For a finite set $S$ of primes of $\QQ$, let $F_S$ denote the maximal algebraic extension of $\QQ$ unramified outside $S$. For an extension $K$ of $\QQ$ contained in $F_S$, we write $G_S(K)=\Gal(F_S\slash K)$. 

 
  \subsection{Rings of periods}\label{overcon}
 
   Let $\overline{\QQ_p}$ be an algebraic closure of $\QQ_p$, and write $\CC_p$ for its $p$-adic completion. Let $\calO_{\CC_p}$ be its ring of integers. Define
   \[ \tilde{\EE}=\varprojlim_{x\rightarrow x^p} \CC_p = \left\{ (x^{(0)},x^{(1)},\dots)\mid \left(x^{(i+1)}\right)^p=x^{(i)}\right\},\]
   and let $\tilde{\EE}^+=\left\{x\in\tilde{\EE}\mid x^{(0)}\in \calO_{\CC_p}\right\}$. If $x=(x^{(i)})$ and $y=y^{(i)}$ are elements of $\tilde{\EE}$, define their sum and product by
   \begin{align*}
    (xy)^{(i)} &= x^{(i)}y^{(i)}\\
    (x+y)^{(i)} &= \lim_{n\rightarrow+\infty}\left(x^{(i+n)}+y^{(i+n)}\right)^{p^n}.
   \end{align*}
   Under these operations, $\tilde{\EE}$ is an algebraically closed field of characteristic $p$. Note that by construction $\tilde{\EE}$ is equipped with a continuous action of $\calG_{\QQ_p}$. Define a valuation on $\tilde{\EE}$ by $v_{\tilde{\EE}}(x)=v_p(x^{(0)})$. Let $\varepsilon=(\varepsilon^{(i)})$ be a fixed element of $\tilde{\EE}$ such that $\varepsilon^{(0)}=1$ and $\varepsilon^{(1)}\neq 1$, and let $\overline{\pi}=\varepsilon-1$. Let $\EE_{\QQ_p}=\FF_p((\overline{\pi}))$, and define $\tilde{\EE}$ to be a separable closure of $\EE_{\QQ_p}$ in $\EE$. Then $\EE$ is equipped with a continuous action of $\calG_{\QQ_p}$, and one can show that $\EE^{H_{\QQ_p}}=\EE_{\QQ_p}$.
  
   Let $\tilde{\AA}=W(\tilde{\EE})$ be the ring of Witt vectors of $\tilde{\EE}$ and
   \[ \tilde{\BB}=\tilde{\AA}[p^{-1}]=\left\{\sum_{k\gg -\infty}p^k[x_k]\mid x_k\in\tilde{\EE}\right\},\]
   where $[x]$ denotes the Teichm\"uller lift of $x\in\tilde{\EE}$. By construction, both rings are equipped with continuous semi-linear actions of a Frobenius operator $\varphi$ and $\calG_{\QQ_p}$. Let $\pi=[\varepsilon]-1$ and $q=\vp(\pi)\slash\pi$, and define $\AA_{\QQ_p}$ to be the completion of $\ZZ_p[[\pi]][\pi^{-1}]$ in the $p$-adic topology. Then $\AA_{\QQ_p}$ is closed under the actions of $\vp$ and $\calG_{\QQ_p}$, and moreover the action of $\calG_{\QQ_p}$ factors through $\Gamma_{\QQ_p}$. Let $\BB$ be the $p$-adic completion of the maximal unramified extension of $\BB_{\QQ_p}=\AA_{\QQ_p}[p^{-1}]$ in $\tilde{\BB}$, and let $\AA=\BB\cap \tilde{\AA}$. These rings are stable under the actions of $\vp$ and $\calG_{\QQ_p}$. For a finite extension $F$ of $\QQ_p$, put $\AA_F=\AA^{H_F}$. If $F_\infty$ is Galois over $\QQ_p$, then $\AA_F$ is equipped with a continuous action of $\calG_F$ which commutes with $\vp$. 
 
   For a $p$-adic representation $T\in\Rep_{\ZZ_p}(\calG_{\QQ_p})$ (resp. $V\in\Rep_{\QQ_p}(\calG_{\Qp})$), define $\DD_F(T)=(\AA\otimes_{\ZZ_p}T)^{H_F}$ (resp. $\DD_F(V)=(\BB\otimes_{\QQ_p}V)^{H_F}$). Then $\DD_F(T)$ (resp. $\DD_F(V)$) is a free finitely generated module over $\AA_F$ of rank $d=\rank_{\ZZ_p}(T)$ (resp. a finite dimensional vector space over $\BB_{\QQ_p}$ of dimension $d=\dim_{\QQ_p}(V)$), equipped with commuting semi-linear actions of $\vp$ and $\Gamma_F$.  Note that $\DD_{F}(T)=\DD_{\QQ_p}(T)\otimes_{\AA_{\QQ_p}}\AA_F$ (resp. $\DD_F(V)=\DD_{\QQ_p}(V)\otimes_{\BB_{\QQ_p}}\BB_F$).
 
   \begin{remark}\label{extendaction}
    If $F_\infty$ is Galois over $\QQ_p$, then the action of $\Gamma_F$ extends to an action of $G_F=\Gal(F_\infty\slash\QQ_p)$. Moreover, the action of $G_F$ commutes with the action of $\vp$.
   \end{remark}

   Every element $x\in\tilde{\BB}$ can be written uniquely of the form $x=\sum_{k\gg-\infty}p^k[x_k]$ with $x_k\in\tilde{\EE}$. For an integer $n\ge0$, define
   \[ \tilde{\BB}^{\dagger,n}=\left\{ x\in\tilde{\BB}\mid \lim_{k\rightarrow+\infty}\left(k+p^{-n}v_{\tilde{\EE}}(x_k)\right)=+\infty\right\}\]
   and let $\BB^{\dagger,n}=\tilde{\BB}^{\dagger,n}\cap\BB$ and $\BB^{\dagger,n}_F=\big(\BB^{\dagger,n}\big)^{H_F}$ for any finite extension $F$ of $\QQ_p$. Also, let $\tilde{\AA}^{\dagger,n}=\{x\in\tilde{\BB}^{\dagger,n}\cap\tilde{\AA}\mid k+ p^{-n}v_{\tilde{\EE}}(x_k)\geq 0 \hspace{1ex} \text{for all $k$}\}$, $\AA^{\dagger,n}=\tilde{\AA}^{\dagger,n}\cap \AA$ and $\AA^{\dagger,n}_F=\big(\AA^{\dagger,n}\big)^{H_F}$. Finally, define $\BB^\dagger=\bigcup_n\BB^{\dagger,n}$, $\AA^\dagger=\bigcup_n\AA^{\dagger,n}$, $\BB_F^\dagger=\bigcup_n\BB_F^{\dagger,n}$ and $\AA_F^\dagger=\bigcup_n\AA_F^{\dagger,n}$. Explicitly, one can describe the ring $\AA^{\dagger,n}_F$ for $n\gg 0$ as follows (c.f.~\cite[Proposition 1.4]{berger02}): there exists $N_F>0$ and $\pi_F\in\AA^{\dagger,N_F}_F$ whose reduction mod $p$ is a uniformizer $\overline{\pi_F}$ of $\EE_F$. Moreover, if $n\geq N_F$, then every element $x\in\BB^{\dagger,n}_F$ can be written as $\sum_{k\in\ZZ}a_k\pi_F^k$, where the $a_k$ are elements in the maximal unramified extension $F'$ of $\QQ_p$ in $F_\infty$, and where the series $\sum_{k\in\ZZ}a_kX^k$ is holomorphic and bounded on the annulus $p^{-1/e_F p^{n-1}(p-1)}\leq\mid X\mid<1$.
 
   Let $F$ be a finite extension of $\QQ_p$. For $V\in \Rep_{\QQ_p}(\calG_{\QQ_p})$, define $\DD_F^{\dagger,r}(V)=\big(\BB^{\dagger,r}\otimes_{\QQ_p}V\big)^{H_F}$ and $\DD_F^{\dagger}(V)=\big(\BB^{\dagger}\otimes_{\QQ_p}V\big)^{H_F}$.  Note that $\DD^\dagger_F(V)=\DD^\dagger_{\QQ_p}(V)\otimes_{\BB^\dagger_{\QQ_p}}\BB^{\dagger}_F$. The main result of~\cite{cherbonniercolmez98} shows that every $p$-adic representation $V$ of $\calG_{\QQ_p}$ is overconvergent, i.e. there exists $r(V)>0$ such that 
   \[ \DD_{\QQ_p}(V)=\BB_{\QQ_p}\otimes_{\BB^{\dagger,r(V)}_{\QQ_p}}\DD^{\dagger,r(V)}(V).\]

   If $V$ is a crystalline representation of $\calG_{\QQ_p}$, then a stronger result is true: $V$ is of finite hight, i.e.  let $\tilde{\BB}^+=W(\tilde{\AA}^+)[p^{-1}]$, $\BB^+=\BB\cap\tilde{\BB}^+$ and $\BB^+_{\QQ_p}=(\BB^+)^{H_{\QQ_p}}$, and define $\DD^+_{\QQ_p}(V)=(\BB^+\otimes_{\QQ_p}V)^{H_{\QQ_p}}$. Then
   \[ \DD_{\QQ_p}(V)=\DD^+_{\QQ_p}(V)\otimes_{\BB^+_{\QQ_p}}\BB_{\QQ_p}.\]
 
  \subsection{The Robba ring} 
  
   We write $\BB_{\rig,\QQ_p}^+$ for the set of $f(\pi)$ where $f(X)\in\QQ_p[[X]]$ converges everywhere on the open unit $p$-adic disc. In particular, $t=\log(1+\pi)\in\Brig$. Let $F$ be a finite extension of $\QQ_p$. For $n\ge0$, define $\BB^{\dagger,n}_{\rig,F}$ to be the completion of $\BB^{\dagger,n}_{F}$ in the Fr\'echet topology, and define the Robba ring $\BB^{\dagger}_{\rig,F}=\bigcup_n \BB^{\dagger,n}_{\rig,F}$. By \cite[Lemme~3.13]{berger02} we have 
   \[ \BB^{\dagger,n}_{\rig,F}=\BB^{\dagger,n}_{\rig,\QQ_p}\otimes_{\BB^{\dagger,n}_{\QQ_p}}\BB^{\dagger,n}_F.\]
   By \cite[Proposition~I.3]{berger03}, we can identify $\BB^{\dagger,n}_{\rig,\QQ_p}$ with the ring of power series
   \[ \BB^{\dagger,n}_{\rig,\QQ_p}=\big\{ f(\pi)\mid f(X)\in\QQ_p\{\{X\}\} \hspace{1ex}\text{converges for $p^{-1/p^{n-1}(p-1)}\leq\mid X\mid<1$}\big\}.\]
   Note that the actions of $\vp$ and $\Gamma_F$ extend to $\BB^+_{\rig,F}$ and $\BB^\dagger_{\rig,F}$. 
 
   The most important application of $\BB^\dagger_{\rig,F}$ is Berger's comparison isomorphism: if $V$ is a crystalline representation of $\calG_F$, we write $\DD_{\cris}(V)$ and $\NN(V)$ for the Dieudonn\'{e} module and the Wach module of $V$ respectively, then there is a canonical isomorphism 
   \begin{equation}\label{comparison}
    \iota: \DD^\dagger_F(V)\otimes_{\BB^\dagger_F}\BB^\dagger_{\rig,F}[t^{-1}]\cong \DD_{\cris}(V)\otimes_{F^{\mathrm{nr}}}\BB^\dagger_{\rig,F}[t^{-1}]
   \end{equation}
   which is compatible \[ \Sel^i(E\slash L_\infty)=\ker\left( H^1(G_S(L_\infty),\Ep)\rightarrow \bigoplus_{v\in S}J^i_v(L_\infty)\right).\]  with the actions of $\calG_F$ and $\vp$. If $V$ is a crystalline representation of $\calG_{\QQ_p}$, then we indeed have a comparison isomorphism
   \[
    \iota: \NN(V)\otimes_{\BB^+_{\QQ_p}}\BB^+_{\rig,\QQ_p}[t^{-1}]\cong \DD_{\cris}(V)\otimes_{\Qp}\BB^+_{\rig,\QQ_p}[t^{-1}].
   \]

  \subsection{The operator $\psi$}

   Note that the extension $\EE$ over $\vp(\EE)$ is inseparable of degree $p$. One can hence define a left inverse $\psi$ of $\vp$ on $\AA$. Explicitly, a basis of $\AA$ over $\vp(\AA)$ is given by $1,1+\pi,\dots,(1+\pi)^{p-1}$. For $x\in\AA$, we may write $x=\sum_{i=0}^{p-1}\vp(x_i)(1+\pi)^i$ where $x_i\in\AA$. We set $\psi(x)=x_0$.
   
   If $F$ is a finite extension of $\QQ_p$ and $V\in\Rep_{\ZZ_p}(\calG_F)$ or $\Rep_{\QQ_p}(\calG_F)$, then $\psi$ extends to a left inverse of $\vp$ $\DD_F(V)$. If $F_\infty$ is Galois over $\QQ_p$, then by Remark~\ref{extendaction} we have an action of $G_F$ on $\DD_F(T)$ which commutes with $\vp$ and hence with $\psi$.

  \subsection{Tate twists}
   
   Let $F$ be a finite extension of $\Qp$. We write $\chi$ for the $p$-cyclotomic character of $\calG_F$. If $m$ is an integer and $V\in\Rep_{\QQ_p}(\calG_F)$ (resp. $T\in\Rep_{\ZZ_p}(\calG_F)$), we denote by $V(m)$ (resp. $T(m)$) the $\calG_F$-representations  $V\otimes_{\Qp}\Qp\cdot e_m$ (resp. $T\otimes_{\Zp}\Zp\cdot e_m$) where $\calG_F$ acts on $e_m$ via $\chi^m$. In particular, we have
   \[
   \DD_{\cris}(V(m))=\DD_{\cris}(V)\otimes t^{-m}e_m,\quad\NN(T(m))=\NN(T)\otimes\pi^{-m}e_m\quad\text{and}\quad\NN(V(m))=\NN(V)\otimes\pi^{-m}e_m.
   \]

  \subsection{The Herr complex}

   We first review some results from $p$-adic Hodge theory. Let $F$ be a finite extension of $\Qp$. Let $T\in\Rep_{\ZZ_p}(\calG_{\Qp})$. Recall the following result from~\cite{herr98} (see also~\cite[\S 2]{cherbonniercolmez99}). Let $\gamma$ be a topological generator of $\Gamma_F$. For $f=\vp$ or $\psi$, define the complex
   \[ \calC^\bullet_{f,\gamma}\big(\DD_F(T)\big): 0\rightarrow \DD_F(T)\rTo^{\alpha_f}\DD_F(T)\oplus\DD_F(T)\rTo^{\beta_f}\DD_F(T)\rightarrow 0,\]
   where $\alpha_f(x)=\big((\gamma-1)x,(f-1)x\big)$ and $\beta_f(x,y)=(f-1)x-(\gamma-1)y$. Denote by $H^i\big(\calC^\bullet_{f,\gamma}(\DD_F(V))\big)$ the $i$-th cohomology group of the complex. 
   
   \begin{theorem}
    For $f=\vp$ or $\psi$, $H^i\big(\calC^\bullet_{f,\gamma}(\DD_F(T))\big)$ is canonically isomorphic to $H^i(F,T)$. In particular, if $(x,y)\in\DD_F(T)^{\oplus 2}$ satisfies $\beta_\vp(x,y)=0$, then the corresponding cohomology class in $H^1(F,T)$ is given by the cocycle
    \[ c_{(x,y)}: \sigma\mapsto \frac{\sigma-1}{\gamma-1}x-(\sigma-1)z,\]
    where $z\in\AA\otimes_{\Zp}T$ is such that $(\vp-1)z=y$.
   \end{theorem}
  
   \begin{theorem}\label{comp}
    We have an $\Lambda(\Gamma_F)$-equivariant isomorphism $H^1_{\Iw}(F,T)\cong\DD_{F}(T)^{\psi=1}$. If $F_\infty$ is Galois over $\Qp$, then the isomorphism is compatible with the action of $G=\Gal(F_\infty\slash \Qp)$.
   \end{theorem}
   \begin{proof}
    See \cite[Th\'eor\`eme II.1.3]{cherbonniercolmez99}.
   \end{proof}

   From now on, we will identify  $H^1_{\Iw}(F,T)$ with $\DD_{F}(T)^{\psi=1}$ under the isomorphism given by Theorem~\ref{comp}.


 \section{The signed Selmer groups}
 
  Let $E$ be an elliptic curve defined over $\QQ$, and fix a prime $p\geq 3$. In this section, we use the theory of $(\vp,\Gamma)$-modules to define signed Selmer groups $\Sel^i(E\slash L(\mu_{p^\infty}))$ for any number field $L$, when $E$ has either good supersingular or good ordinary reduction at $p$. If $E$ has good ordinary reduction at $p$, then Theorem~\ref{sameSelmergroup} shows that $\Sel^2(E\slash L_\infty)$ agrees with the usual Selmer group $\Sel(E\slash L_\infty)$.

  
 \subsection{Good supersingular elliptic curves}\label{supersingular}
 
  Assume throughout this section that $a_p=0$. Let $T_p(E)$ be the Tate module of $E$ at $p$ and write $V=T_pE\otimes_{\Zp}\Qp$, so as a representation of $\calG_{\Qp}$, $V$ is crystalline with Hodge-Tate weights $0,1$. Let $v_1$ be a basis of $\Fil^0\DD_{\cris}(V)$, and extend it to a basis $v_1,v_2$ of $\DD_{\cris}(V)$ such that the matrix of $\vp$ on $\DD_{\cris}(V)$ in this basis is $\begin{pmatrix} 0 & -1 \\ p & 0\end{pmatrix}$. 

  Define $\log^-(1+\pi)=\prod_{i\geq 0}\frac{\vp^{2i}(q)}{p}$ and $\log^+(1+\pi)=\prod_{i\geq 0}\frac{\vp^{2i+1}(q)}{p}$, which are elements of $\BB^+_{\rig,\Qp}$. Since the Hodge-Tate weights of $V$ are non-negative, we have 
  \[ \NN(T)\subset \DD_{\cris}(V)\otimes\BB^+_{\rig,\Qp}\]
  by \cite[Proposition II.2.1]{berger03}, and it follows form Appendice (3) in {\em op.cit.} that a basis $n_1,n_2$ of $\NN(T)$ is given by $\begin{pmatrix} n_1 \\ n_2\end{pmatrix} = M\begin{pmatrix} v_1 \\ v_2\end{pmatrix}$, where 
  \begin{equation} M= \begin{pmatrix} \log^-(1+\pi) & 0 \\ 0 & \log^+(1+\pi) \end{pmatrix}.\label{matrixM}\end{equation}
  Let $K$ be a finite extension of $\Qp$. For $x\in \DD^\dagger_K(T)^{\psi=1}$, we write 
  \[x=x_1v_1+x_2v_2=x_1'n_1+x_2'n_2\] 
  with $x_i\in\BB^{\dagger,N}_{\rig,K}$ and $x_i'\in \AA^{\dagger,N}_{K}$. By \eqref{matrixM}, we have
  \begin{equation}\label{basechange}
   x_1=x_1'\log^-(1+\pi)\quad\text{and}\quad x_2=x_2'\log^+(1+\pi).
  \end{equation}

  \begin{definition}
   For $i=1,2$, let 
   \[ H^1_{\Iw}(K,T)^i=\left\{x\in\DD^\dagger_K(T)^{\psi=1}:\vp(x_i)=-p\psi(x_i)\right\}.\]
   For $n\geq 1$, define $H^1(K_n,T)^i$ to be the image of $H^1_{\Iw}(K,T)^i$ under the natural projection map $H^1_{\Iw}(K,T)\rightarrow H^1(K_n,T)$. 
  \end{definition}
  
  \begin{remark}\label{relationtoColeman}
   As show in~\cite[\S 5.2.1]{leiloefflerzerbes10}, we have $H^1_{\Iw}(\Qp,T)^i=H^1_{\Iw}(\Qp,T)\cap\ker(\underline{\Col}_i)$ for the Coleman maps
   \[ \underline{\Col}_i: H^1_{\Iw}(\Qp,T)\rTo\Lambda(\Gamma)\]
   defined in {\em op.cit.}
  \end{remark}

  \begin{definition}
   Let $H^1_{f,i}(K_n,\Ep)$ be the orthogonal complement of $H^1(K_n,T)^i$ under the Pontryagin duality
   \[ [\sim,\sim]: H^1(K_n,T)\times H^1(K_n,\Ep)\rTo \Qp\slash\Zp\]
   for $i=1,2$.
  \end{definition}

  We now return to the global situation. As above, let $L$ be a finite extension of $\QQ$. For a prime $\nu$ of $L$, denote by $L_\nu$ be the completion of $L$ at $\nu$, and let $L_{\nu,n}=L_\nu(\mu_{p^n})$. Let $S$ be the finite set of primes of $\QQ$ containing $p$, all the primes where $E$ has bad reduction and the infinite prime. Let $i=1,2$. For all $v\in S$, define 
   \[ J^i_v(L_n)= \bigoplus_{w_n\mid v}  \frac{H^1(L_{w_n,n},\Ep)}{H^1_{f,i}(L_{w_n,n},\Ep)},\] 
   where the direct sum is taken over all primes $w_n$ of $L_n$ above $v$ and $H^1_{f,i}(L_{w_n,n},\Ep)=H^1_f(L_{w_n,n},\Ep)$ whenever $v\neq p$. We write $J^i_v(L_\infty)=\varinjlim J^i_v(L_n)$.  
  
   \begin{definition}
    For $i=1,2$, define 
    \[ \Sel^i(E\slash L_\infty)=\ker\left( H^1(G_S(L_\infty),\Ep)\rTo \bigoplus_{v\in S}J^i_v(L_\infty)\right).\] 
   \end{definition}
  
   For $n\geq 0$, we also define
    \[ \Sel^i(E\slash L_n)=\ker\left( H^1(G_S(L_n),\Ep)\rTo \bigoplus_{v\in S}J^i_v(L_{n,v})\right).\] 
   Taking direct limits then gives
    \[\Sel^i(E\slash L_\infty)=\varinjlim \Sel^i(E\slash L_n).\]


  \subsection{Good ordinary elliptic curves}\label{ordinary}
  
   In this section, let $E$ be an elliptic curve defined over $\QQ$ with good ordinary reduction at a prime $p\geq 3$. As above, let $S$ be the finite set of primes of $\QQ$ containing $p$, all the primes where $E$ has bad reduction and the infinite prime. 

   
   \subsubsection{Coleman maps and signed Selmer groups}
   
    We first recall our construction of the signed Selmer groups from \cite{leiloefflerzerbes10}. Let $\bar{\nu}_1,\bar{\nu}_2$ and $\bar{n}_1,\bar{n}_2$ be the bases of $\DD_{\cris}(V(-1))$ and $\NN(V(-1))$, respectively, as defined in \cite[\S 3.2]{leiloefflerzerbes10}. In particular, if $\hat{E}$ denotes the formal group of $E$ and $\hat{V}=T_p\hat{E}\otimes_{\ZZ_p}\QQ_p$, then $\bar{\nu}_1$ is a basis vector of $\DD_{\cris}(\hat{V}(-1))$ and $\bar{n}_1=\bar{\nu}_1$ is a basis of $\NN(\hat{V}(-1))$. If $M'$ is the change of basis matrix with 
    \begin{equation}\label{matrix}
      \begin{pmatrix}\bar{\nu}_1\\\bar{\nu}_2\end{pmatrix}=M'\begin{pmatrix}\bar{n}_1\\\bar{n}_2\end{pmatrix},
    \end{equation}
    then $M'$ is lower triangular, with $1$ and $\frac{t}{\pi}$ on the diagonal. If $x\in\NN(V)$, then there exist unique $x_1,x_2\in\BB_{\Qp}^+$ such that $x=(x_1\bar{n}_1+x_2\bar{n}_2)\otimes\pi^{-1}e_1$. By \eqref{matrix}, we can find unique $x_1',x_2'\in\Brig$ such that
    \begin{equation}\label{secondform}
     x=(x_1'\bar{\nu}_1+x_2'\bar{\nu}_2)\otimes t^{-1}e_1
    \end{equation}
    with $x_2'=x_2$. If $P$ denotes the matrix of $\vp$ with respect to the basis $\bar{n}_1,\bar{n}_2$, then $P$ is upper-triangular. Let $\alpha$ be the unit root of the polynomial $X^2-a_pX+p$, then $P$ is in fact of the form
    \[ P=\begin{pmatrix} \alpha & \star \\
         0 & uq
        \end{pmatrix}\]
    for some $u\in (\BB^+_{\QQ_p})^\times$ which is congruent to $\alpha^{-1}\mod \pi$.

    In \cite{leiloefflerzerbes10}, we have defined two pairs of Coleman maps (with respect to the chosen basis),  
    \begin{align*}
     \Col_i:\NN(V)^{\psi=1} & \rTo(\BB_{\Qp}^+)^{\psi=0}, \hspace{3ex}\text{and} \\
     \underline{\Col}_i:\NN(V)^{\psi=1} & \rTo\Lambda_{\Qp}(\Gamma)
    \end{align*}
    for $i=1,2$ with the following properties: for $x\in\NN(V)^{\psi=1}$, write $x=(x_1\bar{n}_1+x_2\bar{n}_2)\otimes \pi^{-1}e_1$ where $x_1,x_2\in\BB_{\Qp}^+$. Then
    \begin{align}\label{colemanmaps}
     (1-\vp)(x) & =\begin{pmatrix}\Col_1(x) & \Col_2(x) \end{pmatrix}M\begin{pmatrix}\bar{\nu}_1\\\bar{\nu}_2\end{pmatrix}\otimes t^{-1}e_1 \\
     & =\begin{pmatrix}\underline{\Col}_1(x) & \underline{\Col}_2(x) \end{pmatrix}\cdot[(1+\pi)M]\begin{pmatrix}\bar{\nu}_1\\\bar{\nu}_2\end{pmatrix}\otimes t^{-1}e_1
    \end{align}
    where  $M=\frac{t}{\pi q}P^T(M')^{-1}$.  
    
    \begin{definition}
     Let $H^1(\QQ_{p,n},T)^i$ be the image of $\ker(\underline{\Col}_i)\cap\NN(T)^{\psi=1}$ under the natural maps $\NN(T)^{\psi=1}\rightarrow H^1_{\Iw}(\Qp,T)\rightarrow H^1(\QQ_{p,n},T)$ and write $H^1_{f,i}(\QQ_{p,n},\Ep)$ for the exact annihilator of $H^1(\QQ_{p,n},T)^i$ under the Tate pairing.
    \end{definition}
    
    If $E$ is defined over $\QQ$, we can then define the signed Selmer groups $\Sel^i(E/\QQ(\mu_{p^\infty}))$ analogously to the construction when $E$ is supersingular at $p$.
    
    \begin{definition}
     Define
     \[
     \Sel^i(E/\QQ(\mu_{p^n}))=\ker\left(\Sel(E/\QQ(\mu_{p^n}))\rightarrow \frac{H^1(\Qp(\mu_{p^n}),\Ep)}{H^1_{f,i}(\Qp(\mu_{p^n}),\Ep)})\right)
     \]
     where $\Sel(E/\QQ(\mu_{p^n}))$ denotes the usual Selmer group and we define $\Sel^i(E/\QQ(\mu_{p^\infty}))$ to be the direct limit of $\Sel^i(E/\QQ(\mu_{p^n}))$.     
    \end{definition}
    
    We now show that on choosing an appropriate basis, we can describe $\ker(\underline{\Col}_2)$ in a manner similar to the good supersingular case (c.f. Remark~\ref{relationtoColeman}).

   \begin{lemma}\label{basis}
    We can choose $\bar{n}_2$ such that $u=\alpha^{-1}$.
   \end{lemma}
   \begin{proof}
    If we let $\bar{n}_1'=\bar{n}_1$ and $\bar{n}_2'=v\bar{n}_2$ where $v\in(\BB_{\Qp}^+)^\times$, then $\bar{n}_1',\bar{n}_2'$ is also a basis of $\NN(V(-1))$. The matrix of $\vp$ with respect to this basis is of the form
    \[ P'=\begin{pmatrix} \alpha & \star \\
          0 & v^{-1}\vp(v)uq
         \end{pmatrix}.\]
    Note that $\alpha u\equiv1\mod\pi$ implies that $\vp^n(\alpha u)\rightarrow 1$ as $n\rightarrow\infty$. In particular, the product $\prod_{n\ge0}\vp^n(\alpha u)$ converges to an element $u'\in(\BB_{\Qp}^+)^\times$. Since $u'\vp(u')^{-1}=\alpha u$, we deduce that $P'$ is of the required form if we take $v=u'$.
   \end{proof} 
   
  \begin{lemma}\label{firstform}
   With respect to the basis given by Lemma~\ref{basis},
   \[
   \ker(\underline{\Col}_2)=\left\{x\in\NN(V)^{\psi=1}|\vp(x_2)=\alpha x_2\text{ where }x=(x_1\bar{n}_1+x_2\bar{n}_2)\otimes \pi^{-1}e_1\text{ with } x_i\in\BB_{\Qp}^+\right\}.
   \]
  \end{lemma}
  \begin{proof}
   By \eqref{colemanmaps}, we have $\Col_2(x)=\alpha x_2-\vp(x_2)$. Moreover, $M$ is lower triangular with $\alpha\frac{t}{\pi q}$ and $\alpha^{-1}$ on the diagonal. Therefore, $\ker(\underline{\Col}_2)=\ker(\Col_2)$ and we are done.
  \end{proof}
   
  \begin{corollary}\label{description} 
   With respect to the basis given by Lemma~\ref{basis},
   \[
   \ker(\underline{\Col}_2)=\left\{x\in\NN(V)^{\psi=1}|\vp(x_2)=\alpha x_2\text{ where }x=(x_1\bar{\nu}_1+x_2\bar{\nu}_2)\otimes t^{-1}e_1 \text{ with } x_i\in\Brig\right\}.
   \]
  \end{corollary}
  \begin{proof}
   This follows from Lemma~\ref{firstform} and \eqref{secondform}.
  \end{proof}

  Let $L$ be a finite extension of $\QQ$, and let $L_\infty=L(\mu_{p^\infty})$. Using Corollary~\ref{description}, we define $\Sel^2(E\slash L_\infty)$ as follows.  Let $\nu$ be a prime of $L$ above $p$. If $x\in \DD^\dagger_{L_\nu}(T)^{\psi=1}$ then we can use Berger's comparison isomorphism~\eqref{comparison} we can write $x=(x_1 \bar{\nu}_1+ x_2 \bar{\nu}_2)\otimes t^{-1}e_1$ with $x_i\in \BB^\dagger_{\rig,L_\nu}$ as in the supersingular case. Define
  \[
  H^1_{\Iw}(L_\nu,T)^2=\left\{x\in\DD_{L_{\nu}}(T)^{\psi=1}|\vp(x_2)=\alpha x_2\right\}
  \]
  and $H^1(L_{\nu,n},T)^2$ is defined to be the projection of $H^1_{\Iw}(L_\nu,T)^2$ in $H^1(L_{\nu,n},T)$. Let $H^1_{f,2}(L_{\nu,n},\Ep)$ be the exact annihilator of $H^1_{\Iw}(L_\nu,T)^2$.
  
  \begin{definition}
   For all $v\in S$, define 
   \[ J^2_v(L_n)= \bigoplus_{w_n\mid v}  \frac{H^1(L_{w_n,n},\Ep)}{H^1_{f,2}(L_{w_n,n},\Ep)},\] 
   where the direct sum is taken over all primes $w_n$ of $L_n$ above $v$. Here, $H^1_{f,2}(L_{w_n,n},\Ep)=H^1_f(L_{w_n,n},\Ep)$ whenever $v\nmid p$. Define $J^{2}_v(L_\infty)=\varinjlim J^{(2)}_v(L_n)$. Define
   \[ \Sel^2(E\slash L_\infty)=\ker\left( H^1(G_S(L_\infty),\Ep)\rightarrow \bigoplus_{v\in S}J^2_v(L_\infty)\right).\] 
  \end{definition}


  \subsubsection{Properties of $\Sel^2(E\slash L_\infty)$}

   Let us now study the group $H^1_{\Iw}(L_\nu,T)^2$ a bit further. To simplify notation, let $K=L_\nu$.

   \begin{lemma}\label{onlytrivialsolution}
    Let $a\in \ZZ_p^\times$, and assume that $a$ is not a root of unity. If $x\in\BB^\dagger_{\rig,K}$ satisfies 
    \begin{equation}\label{nosolution}
     a x-\vp(x)=0
    \end{equation}
    then $x=0$.
   \end{lemma}
   \begin{proof}
    If $x\in \Brig$, then we can substitute $\pi=e^t-1$ to write $x$ of the form $\sum_{n\geq 0}c_nt^n$ with $c_n\in\QQ_p$. Since $\vp(t)=pt$, it is clear from this description that for any $a\neq 1$ and $x\ne0$, we have $\vp(x)\neq ax$. 
    
    Let $x\in\BB^\dagger_{\rig,\Qp}-\Brig$. Assume that there exists $n\ge 0$ be such that $x\in \BB^{\dagger,n}_{\rig,\QQ_p}$ and $x\not\in \BB^{\dagger,n-1}_{\rig,\QQ_p}$. Then $\vp(x)\in \BB^{\dagger,n+1}_{\rig,\QQ_p}-\BB^{\dagger,n}_{\rig,\QQ_p}$, so $\vp(x)\neq a x$.
    
    If $x\in\BB^{\dagger,n}_{\rig,\QQ_p}$ for all $n\ge0$ and $x\not\in\Brig$, then $\vp^{-1}(x)$ converges in $\BB^+_{\mathrm{dR}}$, so if we write $x=f(\pi)$, then $f(T)$ does not have a pole at $\varepsilon^{(1)}-1$. But $f(T)$ has a pole at $T=0$ as $x\not\in\Brig$, so $f\big( (T+1)^p-1\big)$ has poles at the $\{(\varepsilon^{(1)})^i-1:0\leq i<p\}$. Hence $\vp(x)\neq ax$.
    
    Assume now that $x\in \BB^\dagger_{\rig,K}$ satisfies $\vp(x)=ax$, and that $x\not\in \BB^{\dagger}_{\rig,\QQ_p}$. On replacing $K$ by its Galois closure, if necessary, we may assume that $K\slash\QQ_p$, and hence $K_{\infty}\slash\QQ_{p,\infty}$, are Galois. Let $H=\Gal(K_{\infty}\slash \QQ_{p,\infty})$. Since $\vp$ is $H$-equivariant, $\sigma(x)$ also satisfies~\eqref{nosolution} for all $\sigma\in H$. More generally, if $\sigma_1,\dots,\sigma_i\in H$, then $y=\sigma_1(x_2)\dots\sigma_i(x_2)$ satisfies $a^iy=\vp(y)$. The coefficients of the polynomial
    \[ f(Y)=\prod_{\sigma\in H} (Y-\sigma(x))\]
    are elements in $\BB^{\dagger}_{\rig,\QQ_p}$ which satisfy an equation of the form~\eqref{nosolution}, so they must all be zero by the above argument. But the minimal polynomial of $x$ over $\BB^\dagger_{\rig,\QQ_p}$ divides $f(Y)$, which gives a contradiction. 
   \end{proof}
   
   \begin{remark}\label{Weilnumber}
    The unit root $\alpha$ of the polynomial $X^2-a_pX+p$ is a Weil number of complex absolute value $\sqrt{p}$, so it cannot be a root of unity.
   \end{remark}
   \begin{proposition}
    Let $x\in \DD^\dagger_{K_\nu}(V)^{\psi=1}$, and write $x=(x_1 \bar{\nu}_1+ x_2 \bar{\nu}_2)\otimes t^{-1}e_1$ with $x_i\in \BB^\dagger_{\rig,\QQ_p}$. Then $x\in H^1_{\Iw}(K_\nu,T)^2$ if and only if $x_2=0$.
   \end{proposition}
   \begin{proof}
    Immediate from Lemma~\ref{onlytrivialsolution} and Remark~\ref{Weilnumber}.
   \end{proof}
   
   \begin{corollary}
    $x\in H^1_{\Iw}(K,T)^2$ if and only if $x\in \DD_{K}(\hat{T})^{\psi=1}$.
   \end{corollary}
   \begin{proof}
    It follows immediately from the comparison isomorphism and the fact that $\bar{\nu}_1=\bar{n}_1$ that any $x\in\DD_{K}(T)^{\psi=1}$ which satisfies $\iota(x)=x_1\bar{\nu}_1\otimes t^{-1}e_1$ must indeed lie in $\DD_{K}(\hat{T})$.
   \end{proof}
         
   We can now conclude this section with the following theorem.
   
   \begin{theorem}\label{sameSelmergroup}
    We have $\Sel(E\slash L_\infty)=\Sel^2(E\slash L_\infty)$. 
   \end{theorem}
   \begin{proof}
    Since $L/\QQ$ is finite and $E$ had good ordinary reduction at $p$, we have $V^{H_{L_\nu}}=0$ for all primes $\nu$ of $L$ above $p$. This implies that 
    \[
     H^1_{\Iw}(L_\nu,\hat{T})\otimes_{\Zp}\Qp=\lim_{\leftarrow} H^1_g(L_{\nu,n},T)\otimes_{\Zp}\Qp
    \]
    by \cite[Proposition~0.1]{perrinriou00a} (or \cite[Theorem~A]{berger05}). But $H^1_f(L_{\nu,n},T)=H^1_g(L_{\nu,n},T)$ by \cite[(3.11.2)]{blochkato90}), we have 
    \[
     H^1_{\Iw}(L_\nu,\hat{T})\otimes_{\Zp}\Qp=\lim_{\leftarrow} H^1_f(L_{\nu,n},T)\otimes_{\Zp}\Qp.
    \]
     
    It is clear that the quotients $\DD_K(T)^{\psi=1}\slash\DD_{K}(\hat{T})^{\psi=1}$ and $H^1(L_{\nu,n},T)\slash H^1_f(L_{\nu,n},T)$ are torsion-free over $\Zp$. We can therefore deduce that 
    \[
     H^1_{\Iw}(L_\nu,\hat{T})=\lim_{\leftarrow} H^1_f(L_{\nu,n},T).
    \]
    On taking Pontryagin duals, we have $J_\nu^2(L_\infty)=J_\nu(L_\infty)$, which finishes the proof.
   \end{proof}


\section{An alternative definition of the signed Selmer groups}\label{alter}

 Let $E$ be an elliptic curve defined over $\QQ$ with good supersingular reduction at a prime $p\geq 3$ such that $a_p=0$, and let $L$ be a finite extension of $\QQ$. The main result of this section is Proposition~\ref{samesel} below, which shows that the local conditions at the primes above $p$ in the definition of the signed Selmer groups $\Sel^i(E\slash L_\infty)$ using some ``jumping conditions" similar to those introduced in \cite{kobayashi03}.
 
 \subsection{Preliminary results on $\BB^\dagger_K$} 
   
  Let $K$ be a finite extension of $\Qp$, and let $K'$ be the maximal unramified extension of $\Qp$ contained in $K_\infty$. It is easy to see from the description of the ring $\AA^{\dagger,n}_K$ given in Section~\ref{overcon} that it is complete in the $p$-adic topology.
 
  \begin{lemma}\label{completion}
   For all $n\geq N_K$, $\AA^{\dagger,n}_K$ is the $p$-adic completion of $\mathcal{O}_{K'}[\![\pi_K]\!][\pi_K^{-1}]\cap\AA^{\dagger,n}_K$.
  \end{lemma}
  \begin{proof}
   Note that the condition that $\sum_{k\in\ZZ}a_kX^k$ is holomorphic and bounded above by $1$ on the annulus $p^{-1/e_Kp^{n-1}(p-1)}\leq\mid X\mid<1$ is equivalent to the condition that
   \[ v_p(a_k)+\frac{k}{e_K(p-1)p^{n-1}}\geq 0 \hspace{2ex} \text{and $\rightarrow +\infty$ as $k\rightarrow -\infty$}.\]
  \end{proof}

  \begin{lemma}\label{vanish}
   Let $x\in\AA^{\dagger,N}_K$ where $N\ge N_K$. If $\theta\circ\vp^{-n}(x)=0$ for infinitely many $n\ge N$, then $x=0$.
  \end{lemma}
  \begin{proof}
   Firstly, we assume that $x\in\mathcal{O}_{K'}[\![\pi_K]\!][\pi_K^{-1}]$. We write $\sigma$ for the Frobenious in $\Gal(K'/\Qp)$. Let
   \[
   F(X)=\sum_{m\ge -r}b_mX^m\in\mathcal{O}_{K'}[\![X]\!][X^{-1}]
   \]
   such that $F(\pi_K)=x$. For $i=1,\ldots,[K':\Qp]$, write 
   \[ F_i(X)=\sum_{m\ge -r}\sigma^i(b_m)X^m.\]
   Then $\theta\circ\vp^{-n}(x)=F_i(\pi_{n})$, where $i+n\equiv0\mod[K':\Qp]$ and $\pi_n=\theta\circ\vp^{-n}(\pi_K)$. Therefore, there exists an $i$ such that $F_i$ has infinitely many zeros. But $F_i\in X^{-r}\mathcal{O}_{K'}[\![X]\!]$, so $F_i=0$ by the Weierstrass preparation theorem. This implies that $b_m=0$ for all $m$, so $x=0$. 

   To conclude, note that if $n\ge N$, $\{x\in\AA^{\dagger,N}_K:\theta\circ\vp^{-n}(x)=0\}$ is a closed set of $\AA^{\dagger,N}_K$ under the $p$-adic topology and $\AA^{\dagger,N}_K$ is the $p$-adic completion of $\mathcal{O}_{K'}[\![\pi_K]\!][\pi_K^{-1}]\cap\AA^{\dagger,N}_K$ by Lemma~\ref{completion}.
  \end{proof}

  \begin{lemma}\label{comptrace}
   Let $n\gg0$ and $x\in\BB_{K}^{\dagger,n}$, then
   \[
    \Tr_{K_n/K_{n-1}}\circ \theta\circ\vp^{-n}(x)=\theta\circ\vp^{-n}\circ\Tr_{\BB/\vp(\BB)}(x).
   \]
  \end{lemma}
  \begin{proof}
   We let $n$ be an integer such that $[K_n:K_{n-1}]=p$ and $n\ge a(K)+1$ where $a(K)$ is the integer as in \cite[Proposition~III.2.1]{cherbonniercolmez99}. Write
   \[
    x=\sum_{i=0}^{p-1}[\varepsilon]^i\vp(x_i)
   \]
   where $x_i\in\BB_{K}^{\dagger,n-1}$. Then,
   \[
    \theta\circ\vp^{-n+1}(x_i)\in K_{n-1}
   \]
   for all $i$. Therefore,
   \begin{eqnarray*}
    \Tr_{K_n/K_{n-1}}\circ \theta\circ\vp^{-n}(x)&=&\Tr_{K_n/K_{n-1}}\left(\sum_{i=0}^{p-1}\zeta_{p^n}^i\theta\circ\vp^{1-n}(x_i)\right)\\
    &=&p\theta\circ\vp^{1-n}(x_0).
   \end{eqnarray*}
   But we have $\Tr_{\BB/\vp(\BB)}(x)=p\vp(x_0)$, which finishes the proof.
  \end{proof}
  
  
 \subsection{The local conditions}

  Write $\DD_K^\dagger(T)$ for the overconvergent $(\vp,\Gamma)$-module of $T$ over $K$. It is clear from the definition that
  \[ \DD^\dagger_K(T)=\AA^\dagger_K\otimes_{\AA^+_{\Qp}}\NN(T),\]
  so in particular the basis $n_1,n_2$ of $\NN(T)$ given in Section~\ref{supersingular} is a basis of $\DD^\dagger_K(T)$ over $\AA^\dagger_K$.

  As shown in \cite[Proposition III.3.2]{cherbonniercolmez99}, we have 
  $\DD^\dagger_K(T)^{\psi=1}\subset \DD^{\dagger,N}_K(T)$  for $N\ge N(K,V)$. Fix such an $N$; note that it is not uniquely defined. Let $x\in \DD^\dagger_K(T)^{\psi=1}$. Then as in Section~\ref{supersingular}, we can write 
  \[x=x_1v_1+x_2v_2=x_1'n_1+x_2'n_2\]
  with $x_i\in\BB^{\dagger,N}_{\rig,K}$ and $x_i'\in \BB^{\dagger,N}_{K}$ for $i=1,2$.

  \begin{lemma}\label{triv0}
   Let $x\in \DD^\dagger_K(T)^{\psi=1}$, then
   \[\Tr_{K_n/K_{n-1}}\circ\theta\circ\vp^{-n}(x_1)=-\theta\circ\vp^{2-n}(x_1)\]
   for all odd integers $n\ge N+2$ and
   \[\Tr_{K_n/K_{n-1}}\circ\theta\circ\vp^{-n}(x_2)=-\theta\circ\vp^{2-n}(x_2)\]
   for all even integers $n\ge N+2$.
  \end{lemma}
  \begin{proof}
   By definitions, we have $\vp(\log^+(1+\pi))=\frac{p}{q}\log^-(1+\pi)$ and $\vp(\log^-(1+\pi))=\log^+(1+\pi)$, and that similar relations hold when replacing $\vp$ by $\psi$. The relations \eqref{basechange} therefore imply that
   \begin{eqnarray}
    \vp(x_1)+p\psi(x_1)&=&\left(\vp(x_1')+\psi(qx_1')\right)\log^+(1+\pi);\label{log1}\\
    \vp(x_2)+p\psi(x_2)&=&\left(\vp(x_2')+q\psi(x_2')\right)p/q\log^-(1+\pi).\label{log2}
   \end{eqnarray}
   If $n\ge 2$ is an even integer, then $\theta\circ\vp^{-n}\left(\log^+(1+\pi)\right)=0$. Therefore, \eqref{log1} implies that
   \[
   \theta\circ\vp^{-n+1}(x_1)+\theta\circ\vp^{-n-1}\left(p\vp\circ\psi(x_1)\right)=0.
   \] 
   Recall that $p\vp\circ\psi=\Tr_{\BB/\vp(\BB)}$, so Lemma~\ref{comptrace} implies the first part of the lemma. Similarly, the second half the lemma follows from \eqref{log2} and the fact that 
   \[\theta\circ\vp^{-n}\left(p/q\log^-(1+\pi)\right)=0\]
   for all odd integers $n\ge 3$.
  \end{proof}

  \begin{proposition}\label{des1}
   Let $x\in \DD^\dagger_K(T)^{\psi=1}$, then $x\in H^1_{\Iw}(K,T)^i$ if and only if
   \[\Tr_{K_n/K_{n-1}}\circ\theta\circ\vp^{-n}(x_i)=-\theta\circ\vp^{2-n}(x_i)\]
   for all $n\ge N+2$. 
  \end{proposition}
  \begin{proof}
   If $x\in H^1_{\Iw}(K,T)^i$, then 
   \[
   \vp^2(x_i)=-p\vp\circ\psi(x_i)=-\Tr_{\BB/\vp(\BB)}(x_i).
   \]
   On applying $\theta\circ\vp^{-n}$ to both sides, we have by Lemma~\ref{comptrace} that
   \[
   \theta\circ\vp^{2-n}(x_i)=-\Tr_{K_n/K_{n-1}}\circ\theta\circ\vp^{-n}(x_i)
   \]
   as required.

   Conversely, we assume that 
   \[
   \Tr_{K_n/K_{n-1}}\circ\theta\circ\vp^{-n}(x_i)=-\theta\circ\vp^{2-n}(x_i)
   \]
   for all $n\ge N+2$. Then $\theta\circ\vp^{-n}\left(\vp(x_i)+p\psi(x_i)\right)=0$ by Lemma~\ref{comptrace}. Our assumption implies that $\vp(x_1')+\psi(qx_1')=0$ for $i=1$ and $\vp(x_2')+q\psi(x_2')=0$ for $i=2$ by Lemma~\ref{vanish} and the equations \eqref{log1} and \eqref{log2}. Therefore, we have $x\in H^1_{\Iw}(K,T)^i$ as required.
  \end{proof}

  \begin{remark}
   By Lemma~\ref{triv0}, we can rewrite Proposition \ref{des1} as follows:
   \begin{eqnarray*}
    H^1_{\Iw}(K,T)^1&=&\left\{x\in \DD^\dagger_K(T)^{\psi=1}:\Tr_{K_n/K_{n-1}}\circ\theta\circ\vp^{-n}(x_1)=-\theta\circ\vp^{2-n}(x_1)\text{ for all even }n\ge N+2\right\};\\
    H^1_{\Iw}(K,T)^2&=&\left\{x\in \DD^\dagger_K(T)^{\psi=1}:\Tr_{K_n/K_{n-1}}\circ\theta\circ\vp^{-n}(x_2)=-\theta\circ\vp^{2-n}(x_2)\text{ for all odd }n\ge N+2\right\}.
   \end{eqnarray*}
  \end{remark}

  We can now describe $H^1_{\Iw}(K,T)^i$ as follows.

  \begin{corollary}\label{trace}
   We have
   \begin{eqnarray*}
    H^1_{\Iw}(K,T)^1&=&\left\{x\in \DD_K(T)^{\psi=1}:\exp^*_{K_n}\circ h^1_{\Iw,n}(x)\in K_{n-1}\cdot v_1\text{ for all odd }n\ge N+1\right\},\\
    H^1_{\Iw}(K,T)^2&=&\left\{x\in \DD_K(T)^{\psi=1}:\exp^*_{K_n}\circ h^1_{\Iw,n}(x)\in K_{n-1}\cdot v_1\text{ for all even }n\ge N+1\right\}.
   \end{eqnarray*}
  \end{corollary}
  \begin{proof}
   By \cite[Th\'eor\`eme IV.2.1]{cherbonniercolmez99}, 
   \[ \exp^*_{K_n}\circ h^1_{\Iw,n}(x)= \partial_V\circ \vp^{-n}(x)\]
   for all $n\geq N$. Since
   \[ \partial_V\circ \vp^{-n}(x)=\theta\circ\vp^{-n}(x_1)\vp^{-n}(v_1)+\theta\circ\vp^{-n}(x_2)\vp^{-n}(v_2)\]
    and the image of $\exp^*_{K_n}$ lies in $K_n\otimes{\Fil}^0\DD_{\cris}(V)$, it follows that 
   \[ \exp^*_{K_n}\circ h^1_{\Iw,n}(x)=
    \begin{cases}
    (-1)^mp^{-m}\theta\circ\vp^{-2m}(x_1)v_1 & \text{ if $n=2m\ge N$}, \\
    (-1)^mp^{-m}\theta\circ\vp^{-(2m+1)}(x_2)v_1 & \text{ if $n=2m+1\ge N$}.
    \end{cases}
   \]
   Extend the trace map $\Tr_{K_n/K_{n-1}}$ to the map $\Tr_{K_n/K_{n-1}}\circ\mathrm{id}$ on $K_n\otimes_{\Qp}\DD_{\cris}(V)$. Then, $x\in H^1_{\Iw}(K,T)^1$ if and only if
   \[
   \Tr_{K_n/K_{n-1}}\circ\exp^*_{K_n}\circ h^1_{\Iw,n}(x)=p^{-1}\exp^*_{K_{n-2}}\circ h^1_{\Iw,{n-2}}(x)
   \]
   for all even $n\ge N+2$. But 
   \[
   \Tr_{K_m\slash K_{m-1}}\circ \exp^*_{K_{m}} =\exp^*_{K_{m-1}}\circ {\cor}_{K_m\slash K_{m-1}}
   \]
   for all $m\ge0$, so we deduce that $x\in H^1_{\Iw}(K,T)^1$ if and only if
   \[
   \exp^*_{K_{n-1}}\circ h^1_{\Iw,n-1}(x)=p^{-1}\Tr_{K_{n-1}/K_{n-2}}\circ\exp^*_{K_{n-1}}\circ h^1_{\Iw,n-1}(x)
   \]
   for all even $n\ge N-2$. 
  \end{proof}
 
  As a important consequence, we can characterise $H^1_{\Iw}(K,T)^i$ completely in terms of the conditions on the finite levels:
 
  \begin{corollary}\label{sameinv}
   For $i\in\{1,2\}$ and $n\geq N+1$, define
   \[ H_N^1(K_n,T)^{(i)}=\left\{x\in H^1(K_n,T):\Tr_{K_n\slash K_m}\circ\exp^*_{K_n}(x)\in K_{m-1}\cdot v_1\hspace{2ex}\text{for all $m\in S_{N,i}^n$}\right\}.\]
   where $S_{N,i}^n$ is given by
   \begin{eqnarray*}
    S_{N,1}^n&=&\left\{m\in[N+1,n]:m\text{ odd }\right\},\\
    S_{N,2}^n&=&\left\{m\in[N+1,n]:m\text{ even }\right\}.
   \end{eqnarray*}
   Then $H^1_{\Iw}(K,T)^i=\varprojlim H_N^1(K_n,T)^{(i)}$.
  \end{corollary}
  \begin{proof}
   Immediate from Corollary~\ref{trace}.
  \end{proof}

 {\bf Notation.} Let $F$ be a finite extension of $\Qp$. For an integer $n\ge1$, we write $F_n^{(0)}=\ker(\Tr_{F_n/F_{n-1}})$. Then we have
  \begin{equation}\label{pro}
  F_n=F\oplus\bigoplus_{i=1}^nF_i^{(0)}.
  \end{equation}

  \begin{lemma}\label{3des}
   Let $n\ge N+1$ be an integer, then
   \[
    H_N^1(K_n,T)^{(i)}=\left(\exp^*_{K_n}\right)^{-1}\left(K_N\oplus\bigoplus_{m\in S_{N,i'}^n}K_{m}^{(0)}\cdot v_1\right)
   \]
   where $\{i'\}=\{1,2\}\setminus\{i\}$.
  \end{lemma}
  \begin{proof}
   Let $x\in K_n$. By definition, the projection of $x$ under \eqref{pro} into $K_m^{(0)}$ is zero if and only if $\Tr_{K_n/K_{m}}x\in K_{m-1}$. Hence the result.
  \end{proof}
   
  From now on, we make the following assumption. 
  
  \begin{assumption}\label{assumptionptorsion}
   $E(K_\infty)$ has no $p$-torsion.
  \end{assumption}
 
  Note that this is satisfied for example when $[K:\Qp]$ is a power of $p$. 
  \begin{remark}
   Assumption~\ref{assumptionptorsion} implies that the natural map $H^1(K_n,T)\rTo H^1(K_n,V)$ is injective for all $n\ge0$. In particular, we may embed $H^1(K_n,T)$ into $H^1(K_n,V)$ and consider the former as a lattice inside the latter. 
  \end{remark}

  \begin{proposition}\label{describingortho}
   Let $H^1_{f,N,(i)}(K_n,T)$ be the exact annihilator of $H_N^1(K_n,T)^{(i)}$ under the Tate pairing. Then 
   \[
     H_{f,N,(i)}^1(K_n,T) = H^1(K_n,T)\cap \exp_{K_n}\left( \bigoplus_{m\in S_{N,i}^n}K_{m}^{(0)} \otimes\DD_{\cris}(V) \right)
   \]
  \end{proposition}
  \begin{proof}
By Lemma~\ref{3des}, we have
   \begin{equation}\label{E}
    H^1_{f,N,(1)}(K_n,T)=\left(\left(\exp^*_{K_n}\right)^{-1}\left(K_N\oplus\bigoplus_{m\in S_{N,i'}^n}K_{m}^{(0)}\cdot v_1\right)\right)^{\bot_{[,]}}
   \end{equation}
where $(\star)^{\bot_{[,]}}$ denotes the exact annihilator of $\star$ under the pairing $[\sim,\sim]$. But
\[
[\exp_{K_n}(\sim),\sim]=\Tr_{K_n/\Qp}\langle\sim,\exp^*_{K_n}(\sim)\rangle.
\] 
where $\langle\sim,\sim\rangle$ is the pairing
\[
\langle\sim,\sim\rangle:\Big(K_n\otimes\DD_{\cris}(V)\Big)\times\Big(K_n\otimes\DD_{\cris}(V)\Big)\rightarrow K_n.
\] 
Therefore,
\begin{equation}\label{ortho}
x\in\left(\left(\exp^*_{K_n}\right)^{-1}\left(\star\cdot v_1\right)\right)^{\bot_{[,]}}\quad\text{if and only if}\quad x\in\exp_{K_n}\left((\star)^\bot\otimes\DD_{\cris}(V)\right)
\end{equation}
     where $(\star)^\bot$ denotes the orthogonal complement of $\star$ under the pairing
     \begin{align*}
      K_n\times K_n &\rightarrow \Qp\\
      (x,y) &\mapsto \Tr_{K_n/\Qp}(xy).
     \end{align*}
    By linear algebra, we have
\[
\left(K_N\oplus\bigoplus_{m\in S_{N,i'}^n}K_{m}^{(0)}\right)^\bot=\bigoplus_{m\in S_{N,i}^n}K_{m}^{(0)}.
\]
Hence the result on combining \eqref{E} with \eqref{ortho}.
  \end{proof}

Recall that the exponential map $\exp_{K_n}$ gives an isomorphism
\[
\exp_{K_n}:K_n\otimes\DD_{\cris}(V)/\Fil^0\DD_{\cris}(V)\rTo H^1_f(K_n,V).
\]
We write $\exp_{K_n}^{-1}$ for its inverse.

By \eqref{pro}, we may define a projection map
\[
P_{N,i}^n:K_n\rTo K_N\oplus\bigoplus_{m\in S_{N,i'}^n}K_m^{(0)}.
\]
We can then rewrite Proposition~\ref{describingortho} as follows. 

\begin{corollary}
For $i=1,2$, we have
\[
 H_{f,N,(i)}^1(K_n,T) = \left\{x\in H_{f}^1(K_n,T):\left(P_{N,i}^n\otimes{\rm id}\right)\circ\exp_{K_n}^{-1}(x)=0\right\}.
\]
\end{corollary}
\begin{proof}
Note that
\[
K_n=\left(K_N\oplus\bigoplus_{m\in S_{N,i'}^n}K_m^{(0)}\right)\oplus\left( \bigoplus_{m\in S_{N,i}^n}K_{m}^{(0)}\right).
\]
Therefore, 
\[
\left(P_{N,i}^n\otimes{\rm id}\right)\circ\exp_{K_n}^{-1}(x)=0\quad\text{if and only if}\quad\exp_{K_n}^{-1}(x)\in \bigoplus_{m\in S_{N,i}^n}K_{m}^{(0)}\otimes\DD_{\cris}(V)/\Fil^0\DD_{\cris}(V).
\]
\end{proof}

   Recall that we have a commutative diagram
  
  
  \begin{diagram}
   {\rm tan}(\hat{E}\slash K_n) & \lTo^{\supset \hspace{3ex}} & \log_{\hat{E}}(\calO_{K_n}) & \lTo^{\log_{\hat{E}}}& \hat{E}(\calO_{K_n}) & \rTo^{\subset} & \hat{E}(\calO_{K_n})\otimes \Qp \\
   \dTo^{\cong}_i           & & & & &               & \dTo^{\delta} \\
   K_n\otimes \DD_{\cris}(V)\slash \Fil^0\DD_{\cris}(V) & & & \rTo^{\exp_{K_n}} & & & H^1(K_n,V)
  \end{diagram}
  where $\delta$ is the Kummer map. If we identify the image of $\hat{E}(\calO_{K_n})$ under $\delta$ with $H^1_f(K_n,T)$, we have:
  
   \begin{corollary}\label{rewrite}
    For $i\in\{1,2\}$, the image of $H_{f,N,(i)}^1(K_n,T)$ inside $\hat{E}(\calO_{K_n})$ coincides with
    \begin{align*}
    \hat{E}_N^i(\calO_{K_n})   :=& \left\{x\in\hat{E}(\calO_{K_n}): \Tr_{K_n/K_m}x\in \hat{E}(\calO_{K_{m-1}})\text{ for all }m\in S_{N,i'}^n \text{ and }\Tr_{K_n/K_N}x=0\right\}\\
    =&\left\{x\in\hat{E}(\calO_{K_n}):P_{N,i}^n\circ\log_{\hat{E}}(x)=0\right\}.
    \end{align*}
   \end{corollary}
   \begin{proof}
    By the comutative diagram above and Proposition~\ref{describingortho}, we have 
    \[
     \delta(x)\in H_{f,N,(i)}^1(K_n,T)\quad\text{if and only if}\quad i\circ\log_{\hat{E}}(x)\in \bigoplus_{m\in S_{N,i}^n}K_{m}^{(0)}\otimes\DD_{\cris}(V)/\Fil^0\DD_{\cris}(V).
    \] 
    Since $\log_{\hat{E}}$ is injective (by Assumption~\ref{assumptionptorsion}) and compatible with the trace maps, we are done.
   \end{proof}


 \subsection{Signed Selmer groups revisited}
 
  We now return to the global situation as set up at the beginning of Section~\ref{alter}. Throughout this section, we continue to assume that Assumption~\ref{assumptionptorsion} holds at all the primes of $L$ above $p$. We define the signed Selmer groups of $E$ over $L_\infty$ using the ``jumping conditions" we obtained in the previous section.

  \begin{definition}
   Let $L$ be a number field and $N$ is an integer such that $N\ge N(L_w,V)$ for all primes $w$of $L$ above $p$. For $i=1,2$, we define the Selmer groups
   \[
   \Sel^{(i)}_N(E/L_n)=\ker\left(\Sel(E/L_n)\rTo \bigoplus_{w\mid p}\frac{H^1(L_{n,w},\Ep)}{\hat{E}_N^i(\calO_{L_{n,w}})\otimes\Qp/\Zp }\right)
   \]
   for $n\ge N+1$. Moreover, we define
   \[
    \Sel^{(i)}_N(E/L_\infty)=\varinjlim_{n\ge N+1}\Sel^{(i)}_N(E/L_n)=\ker\left(\Sel(E/L_\infty)\rTo \bigoplus_{\omega\mid p} \frac{H^1(L_{\infty,\omega},\Ep)}{\hat{E}_N^i(\calO_{L_{\infty,\omega}})\otimes\Qp/\Zp} \right)
   \]
   where $\hat{E}_N^i(\calO_{L_{\infty,\omega}})=\varinjlim \hat{E}_N^i(\calO_{L_{n,\omega\cap L_n}})$.
  \end{definition}

  \begin{lemma}\label{duals}
   Let $K$ be a finite extension of $\Qp$ and $n\ge N(K,V)$. For $i=1,2$, the exact annihilator of $H_N^1(K_n,T)^{(i)}$ under the Pontryagin duality
   \[
    [\sim,\sim]: H^1(K_n,T)\times H^1(K_n,\Ep)\rTo\Qp\slash\Zp
   \]
   is isomorphic to $H^1_{f,N,(i)}(K_n,T)\otimes\Qp/\Zp$ for $i=1,2$.
  \end{lemma}
  \begin{proof}
   This essentially follows from \cite[proofs of Lemma~8.17 and Proposition~8.18]{kobayashi03}. By definition, we have an exact sequence
   \[
    0\rTo H_N^1(K_n,T)^{(i)} \rTo H^1(K_n,T)\rTo\Hom\left(H^1_{f,N,(i)}(K_n,T),\Zp\right).
   \]
   On taking Pontryagin duals, we obtain a second exact sequence
   \[
    H^1_{f,N,(i)}(K_n,T)\otimes\Qp/\Zp\rTo H^1(K_n,\Ep)\rTo H_N^1(K_n,T)^{(i),\vee}\rTo0.
   \]
   Therefore, it remains to show that the first map above is injective. But $[px,y]=p[x,y]$ for all $x,y\in H^1(K_n,T)$. This implies that if $x\in H^1(K_n,T)$ such that $px\in H^1_{f,N,(i)}(K_n,T)$, then $x\in H^1_{f,N,(i)}(K_n,T)$.
  \end{proof}

  \begin{corollary}\label{orthdual}
   Let $K$ be a finite extension of $\Qp$ and $n\ge N(K,V)$. For $i=1,2$, the exact annihilator of $H_N^1(K_n,T)^{(i)}$ under the Pontryagin duality
   is isomorphic to $\hat{E}_N^i(\calO_{K_{n}})\otimes\Qp/\Zp$ for $i=1,2$.
  \end{corollary}
  \begin{proof}
   This follows immediately from Corollary~\ref{rewrite} and Lemma~\ref{duals}.
  \end{proof}

  \begin{proposition}\label{samesel}
   The two definitions of signed Selmer groups coincide, namely, 
   \[
   \Sel^{(i)}_N(E/L_\infty)=\Sel^i(E/L_\infty)
   \]
   for $i=1,2$.
  \end{proposition}
  \begin{proof}
   It suffices to show that for any finite extensions $K$ of $\Qp$, we have
   \[
    \varinjlim_n\frac{H^1(K_{n},\Ep)}{H^1_{f,i}(K_{n},\Ep)}\cong\varinjlim_{n\ge N+1}\frac{H^1(K_{n},\Ep)}{\hat{E}_N^i(\calO_{K_{n}})\otimes\Qp/\Zp}
   \]
   where $N\ge N(K,V)$. On taking Pontryagin duals, this is equivalent to showing
   \[
    \varprojlim H^1(K_n,T)^{i}\cong\varprojlim H^1_N(K_n,T)^{(i)}
   \]
   by Corollary~\ref{orthdual}. Therefore, we are done by Corollary~\ref{sameinv}.
  \end{proof}
  

\section{The supersingular $\mathfrak{M}_H(G)$-conjecture}\label{MHG}

  Throughout this section, we assume that $E$ is an elliptic curve over $\QQ$ with good supersingular reduction at a prime $p\geq 3$ and $a_p=0$. Let $L_\infty$ be a $p$-adic Lie extension of $\QQ$ containing $\QQ(\mu_{p^\infty})$, so $G=\Gal(L_\infty\slash\QQ)$ is a compact $p$-adic Lie group of finite rank. Let $H=\Gal(L_\infty\slash \QQ(\mu_{p^\infty}))$.  Choose a sequence of finite extensions $L_m$ of $\QQ$ such that $L_\infty=\varinjlim L_m$ and $L^{(m)}_\infty=L_m(\mu_{p^\infty})$ is Galois over $\QQ$ for all $m\geq 0$. Recall that for $i=1,2$, we have defined $\Sel^i(E/L^{(m)}_\infty)$ in Section~\ref{supersingular}. This allows to make the following definition. 
  
  \begin{definition}
   For $i=1,2$, we define $\Sel^i(E/L_\infty):=\varinjlim_m\Sel^i(E/L^{(m)}_\infty)$ for $=1,2$ and write
   \[ X_i(E\slash L_\infty)=\Hom_{\mathrm{cts}}\left(\Sel^i(E\slash L_\infty),\QQ_p\slash\ZZ_p\right). \]
  \end{definition}
  
  \begin{definition}
   Denote by $\mathfrak{M}_H(G)$ the category of finitely generated $\Lambda(G)$-modules $M$ for which $M\slash M(p)$ is finitely generated over $\Lambda(H)$. Here $M(p)$ denotes the $p$-torsion part of $M$. 
  \end{definition}
 
  The $\MHG$-conjecture in \cite{cfksv} states that the Pontryagin dual of the Selmer group of $E$ over $L_\infty$ is an element of $\MHG$ if $E$ has good ordinary reduction at $p$. We therefore analogously propose the following conjecture.
  
  \begin{conjecture}\label{con}
   Let $E$ be an elliptic curve over $\QQ$ with good supersingular reduction at $p$ and $a_p=0$. Let $L_\infty$ be a $p$-adic Lie extension of $\QQ$ containing $\QQ(\mu_{p^\infty})$. Define the Galois groups $G=\Gal(L_\infty\slash\QQ)$ and $H=\Gal(L_\infty\slash \QQ(\mu_p))$. Then $X_i(E\slash L_\infty)\in\mathfrak{M}_H(G)$ for $i=1,2$.
  \end{conjecture}
  

 \section{Difficulties}\label{cyclotomicshortexact}
 
  To simplify the notation, let $\QQ_\infty=\QQ(\mu_{p^\infty})$. In order to support Conjecture~\ref{con}, we tried to prove the following result:
      
  \begin{conjecture}\label{support}
   Let $E$ be an elliptic curve over $\QQ$ with good supersingular reduction at $p$ and $a_p=0$. Let $L_\infty$ be a $p$-adic Lie extension of $\QQ$ containing $\QQ(\mu_{p^\infty})$, and let $H=\Gal(L_\infty\slash \QQ(\mu_{p^\infty}))$. Assume also that $E(L_{v,\infty})$ has no $p$-torsion for any prime $v$ of $L$ above $p$.  Then for $i=1,2$, the kernel and cokernel of the restriction map 
   \[
   \Sel^i(E\slash \QQ_\infty)\rTo \Sel^i(E\slash L_\infty)^H
   \]
   are cofinitely generated $\ZZ_p$-modules. 
  \end{conjecture}
  
  Recall that $\Sel^i(E\slash \QQ_\infty)$ is $\Lambda(\Gamma)$-cotorsion (\cite[Theorem~1.2]{kobayashi03}). Assume $H$ is pro-$p$, and that Conjecture~\ref{support} holds. If $\Sel^i(E\slash\QQ_\infty)$ is a cofinitely generated $\ZZ_p$-module, which is equivalent to the vanishing of the $\mu$-invariant of $X_i(E/\slash\QQ_\infty)$ as conjectured in \cite[\S10]{kobayashi03}, then we can apply Nakayama's lemma (c.f. for example \cite[Theorem 2.6]{coateshowson01}) to deduce that $X_i(E\slash L_\infty)$ is finitely generated over $\Lambda(H)$.
  
  In this section, we will explain some of the difficulties that we encountered when trying to prove Conjecture~\ref{support} when $L_\infty$ is a finite extension of $\QQ_\infty$. We first establish a preliminary result (Corollary~\ref{surjectivity}), which allows us to study a fundamental diagram (see the beginning of Section~\ref{fundamental}) analogous to the ordinary case. 
  
   
 \subsection{Analysis of Poitou-Tate exact sequences}
  
  Write $S_f$ for the set of finite places of $S$ and let $I_v^i$ be as defined in Section~\ref{supersingular}. By \cite[\S A.3]{perrinriou95}, there are two exact sequences
  \begin{align}
   &0\rightarrow\Sel^i(E/\QQ(\mu_{p^n}))\rTo H^1(G_S(\QQ(\mu_{p^n})),\Ep)\rTo\bigoplus_{v\in S_f}J_v^i(\QQ(\mu_{p^n}))\rTo H^1_i(\QQ,T)^\vee\rTo\cdots\label{firstpt}\\
   &0\rightarrow H^1_i(\QQ(\mu_{p^n}),T)\rTo^{f_n}H^1(G_S(\QQ(\mu_{p^n})),\Ep)\rTo^{g_n}\bigoplus_{v\in S_f}I_v^i(\QQ(\mu_{p^n}))\rTo\cdots\label{secondpt}
  \end{align}
  where $H^1_i(\QQ(\mu_{p^n}),T)$ is defined by
  \[
   \ker\Big(H^1(\QQ(\mu_{p^n}),T)\rTo\prod_{v\in S_f}I^i_v(\QQ(\mu_{p^n}))\Big)
  \]
  and $M^\vee$ denotes the Pontryagin dual of $M$.
  
  \begin{lemma}
   The natural map
   \[
   H^1\big(G_S(\QQ_\infty),\Ep\big)\rTo\bigoplus_{v\in S_f} J_v^i(\QQ_\infty)
   \]
   is surjective.
  \end{lemma}
  \begin{proof}
   On taking inverse limit, we have
   \[
    \varprojlim_n\Big( \bigoplus_{v\in S_f}I_v^i(\QQ(\mu_{p^n}))\Big) = \frac{H^1_{\Iw}(\Qp,T)}{H^1_{\Iw}(\Qp,T)^i}
   \]
   and $\displaystyle\varprojlim_n g_n$ is injective by \cite[Theorem~7.3]{kobayashi03}. Therefore, on taking inverse limit in \eqref{secondpt}, we have 
   \[\varprojlim_n H^1_i(\QQ(\mu_{p^n}),T)=0,\]
   which implies that $\displaystyle\varinjlim_n H^1_i(\QQ(\mu_{p^n}),T)^\vee=0$. Therefore, on taking direct limit in \eqref{firstpt}, we have an exact sequence
   \[
    0\rTo\Sel^i(E/\QQ_\infty)\rTo H^1\big(G_S(\QQ_\infty),\Ep\big)\rTo\bigoplus_{v\in S_f} J_v^i(\QQ_\infty)\rTo 0
   \]
   and we are done.
  \end{proof}

  \begin{corollary}\label{surjectivity}
   The natural map
   \[
   H^1\big(G_S(\QQ_\infty),\Ep\big)\rTo\bigoplus_{v\in S} J_v^i(\QQ_\infty)
   \]
   is surjective.
  \end{corollary}
  \begin{proof}
   This follows as $J_v^i(\QQ_\infty)=0$ for $p\neq2$ if $v$ is an infinite prime.
  \end{proof}  
  

  \subsection{The fundamental diagram}\label{fundamental}
   
   We attempted to prove Conjecture~\ref{support} by studying the following commutative diagram, which we call {\it the fundamental diagram}.
   
   \begin{diagram}
    0 & \rTo & \Sel^i(E\slash L_\infty)^{H} & \rTo & H^1(G_S(L_\infty),\Ep)^{H} & \rTo & \bigoplus_{v\in S}J_v^i(L_\infty)^{H} & &\\   
      &      & \uTo^{\alpha}      &      &   \uTo^{\beta}              &      & \uTo^{\gamma=(\gamma_v)} & & \\
    0 & \rTo & \Sel^i(E\slash \QQ_\infty) & \rTo & H^1(G_S(\QQ_\infty,\Ep) & \rTo & \bigoplus_{v\in S} J_v^i(\QQ_\infty) & \rTo & 0    
   \end{diagram}
   where the $J_v^i$ are as defined in Section~\ref{supersingular}. Applying the snake lemma gives a long exact sequence
  \[ 0 \rTo \ker(\alpha) \rTo \ker(\beta) \rTo \ker(\gamma) \rTo \coker(\alpha)\rTo \coker(\beta).\]
  In order to prove Conjecture~\ref{support}, it is therefore sufficient to show that the kernel and cokernel of the map $\beta$ and the kernel of $\gamma$ are cofinitely generated $\ZZ_p$-modules. The results for $\beta$ and for $\gamma_v$, $v\nmid p$, are easy consequences of the inflation-restriction exact sequences (c.f. \cite{coatessujatha00}). 

  
  The main difficulty is the study of the kernel of the local restriction map $\gamma_v$ when $v\mid p$. Let $K$ be the completion of $L$ at such a prime, and write $\mathcal{H}$ for the Galois group $\Gal(K_\infty\slash \QQ_{p,\infty})$. In order to prove Conjecture~\ref{support}, we may use the local conditions from Section~\ref{alter} and attempt to show that the kernel of the map
  \[
   \frac{H^1(\QQ_{p,\infty},\Ep)}{\hat{E}^i_N(\calO_{\QQ_{p,\infty}})\otimes\Qp\slash\Zp}\rTo\left(\frac{H^1(K_\infty,\Ep)}{\hat{E}^i_N(\calO_{K_\infty})\otimes\Qp\slash\Zp}\right)^\mathcal{H}
  \]
  is a cofinitely generated $\Zp$-module. Consider the following commutative diagram: 
  \begin{diagram}
   0&\rTo& \hat{E}^i_N(\calO_{\QQ_{p,\infty}})\otimes\Qp\slash\Zp&\rTo& H^1(\QQ_{p,\infty},\Ep)&\rTo&  \frac{H^1(\QQ_{p,\infty},\Ep)}{\hat{E}^i_N(\calO_{\QQ_{p,\infty}})\otimes\Qp\slash\Zp}&\rTo& 0\\
    &      & \dTo             &      & \dTo              &      &   \dTo                   & & \\  
   0&\rTo& \left( \hat{E}^i_N(\calO_{K_\infty})\otimes\Qp\slash\Zp\right)^\mathcal{H}&\rTo&\left( H^1(K_{\infty},\Ep)\right)^\mathcal{H}&\rTo&  \left(\frac{H^1(K_{\infty},\Ep)}{\hat{E}^i_N(\calO_{K_\infty})\otimes\Qp\slash\Zp}\right)^\mathcal{H}&& 
  \end{diagram}
  where the vertical maps are restrictions. The the first two maps are injective by Assumption~\ref{assumptionptorsion}. By the snake lemma, the kernel of the third map is bounded by the cokernel of the first, so it is sufficient to show that the cokernel of the restriction map
  \[ \hat{E}^i_N(\calO_{\QQ_{p,\infty}})\otimes\Qp\slash\Zp\rTo \left( \hat{E}^i_N(\calO_{K_\infty})\otimes\Qp\slash\Zp\right)^\mathcal{H}\]
  is a cofinitely generated $\Zp$-module. By taking $\mathcal{H}$-cohomology of the short exact sequence
  \[
   0\rTo \hat{E}^i_N(\calO_{K_\infty})\rTo\hat{E}^i_N(\calO_{K_\infty})\otimes \Qp\rTo\hat{E}^i_N(\calO_{K_\infty})\otimes\Qp/\Zp\rTo0,
  \]
  we may reduce the validity of Conjecture~\ref{support} to the following conjecture.
  
  \begin{conjecture}\label{cofcoh}
   The cohomological group $H^1\left(\mathcal{H},\hat{E}^i_N(\calO_{K_\infty})\right)$ is a cofinitely generated $\Zp$-module.
  \end{conjecture}

  Note that it is shown in \cite[Theorem~3.1]{coatesgreenberg} that $H^1\left(\mathcal{H},\hat{E}(\calO_{K_\infty})\right)=0$. It therefore might be possible to prove Conjecture~\ref{cofcoh} by showing that an exact sequence similar to \cite[(8.22)]{kobayashi03} holds, e.g., to give a bound on the cokernel of the last map of 
  \[
   0\rTo\hat{E}(\calO_{K_N})\rTo^{x\mapsto x\oplus x}\hat{E}^1_N(\calO_{K_\infty})\oplus \hat{E}^2_N(\calO_{K_\infty})\rTo^{x\oplus y\mapsto x-y}\hat{E}(\calO_{K_\infty}).
  \]
  Unfortunately, this does not seem to be straightforward as far as we can see.
  

\bibliographystyle{../amsalphaurl} 
\bibliography{../references}   

\end{document}